\documentclass[english, 10pt]{amsart}
\usepackage{amsmath}
\usepackage{graphicx}
\usepackage{amssymb}
\usepackage{amsfonts, amsthm}
\usepackage[colorlinks=true,linkcolor=blue,citecolor=blue,urlcolor=blue]{hyperref}  
\DeclareMathOperator{\inte}{int}
\newtheorem{assumption}{Assumption}

\newcommand{\auskommentieren}[1]{}

\newcommand{\beq}{\begin{equation}}
\newcommand{\eeq}{\end{equation}}
\newcommand{\bea}{\begin{equation}\begin{aligned}}
\newcommand{\eea}{\end{aligned}\end{equation}}

\newtheorem{theorem}{Theorem}[section]

\newtheorem{lemma}[theorem]{Lemma}

\theoremstyle{definition}
\newtheorem{remark}[theorem]{Remark}
\newtheorem{definition}[theorem]{Definition}

\usepackage{array} 
\usepackage{amssymb}
\usepackage{latexsym}
\usepackage{epic}
\usepackage{curves}
\usepackage{fancyhdr}
\usepackage{xspace}

\numberwithin{equation}{section}

\DeclareMathOperator{\dist}{dist}

\title[two-scale method for Monge-Amp\`{e}re type equations]{Definition and certain convergence properties of a
two-scale method for Monge-Amp\`{e}re type equations}

\author{Heiko Kr\"oner}
\address{Universit\"at Duisburg-Essen, Fakult\"at f\"ur Mathematik, Mathematikcarr\'ee, Thea-Leymann-Stra\ss e 9, 45127 Essen, Germany}
\email{heiko.kroener@uni-due.de}

\begin{document}

\begin{abstract}
The  Monge-Amp\`{e}re equation arises in the theory of optimal transport. When more complicated cost functions are involved in the optimal transportation problem, which are motivated e.g. from economics, the corresponding equation 
for the optimal transportation map becomes a Monge-Amp\`{e}re type equation. Such Monge-Amp\`{e}re type equations are a topic of current research from the viewpoint of mathematical analysis. From the numerical point of view there is a lot of current research for the Monge-Amp\`{e}re equation itself and rarely for the more general Monge-Amp\`{e}re type equation. Introducing the notion of discrete $Q$-convexity as well as 
specifically designed  barrier functions this purely theoretical paper extends the very recently studied two-scale method approximation of the Monge-Amp\`{e}re itself \cite{NochettoNtogkasZhang2018} to the more general 
Monge-Amp\`{e}re type equation as it arises e.g. in \cite{PhillippisFigalli2013} in the context of Sobolev regularity.
\end{abstract}
\maketitle
\tableofcontents
\section{Introduction}
\label{intro}
The starting point and motivation on the very basic level for our paper is Monge's transportation problem which is formulated in \cite{Monge1781}. Here we recall it by using its formulation 
from the introduction of \cite{FigalliKimMcCann2013}.
Let $0\le f^{-}, f^{+}\in L^1(\mathbb{R}^n)$ be probability densities with respect to the Lebesgue measure $L^n$ on $\mathbb{R}^n$
and $c:\mathbb{R}^n\times \mathbb{R}^n\rightarrow [0, +\infty]$ a cost function. Then Monge's optimal transport problem consists in finding a mapping $G:\mathbb{R}^n\rightarrow \mathbb{R}^n$ which pushes
$d\mu^{+}=f^{+}dL^n$ forward to $d\mu^{-}=f^{-}dL^n$ and which minimizes the expected transportation cost
\begin{equation}
\inf_{G_{\#}\mu^{+}=\mu^{-}} \int_{\mathbb{R}^n}c(x, G(x)) d \mu^{+}(x)
\end{equation}
where $G_{\#}\mu^{+}=\mu^{-}$ means $\mu^{-}[Y]=\mu^{+}[G^{-1}(Y)]$ for each Borel set $Y \subset \mathbb{R}^n$.
It is of interest under which conditions such a map $G$ exists and, furthermore, under which conditions such a map
has a certain classical or Sobolev regularity, for details concerning this we refer to \cite{FigalliKimMcCann2013}
and \cite{PhillippisFigalli2013} and the references to the literature therein. 
Under appropriate assumptions which are not stated here explicitly it turns out that the optimal transportation map $u$ satisfies the following Monge-Amp\`{e}re type equation
\begin{equation} \label{1}
\begin{aligned}
\det \left(D^2u-A(x,Du)\right) = &f \quad \text{in }\Omega  \\
u=& g, \quad \text{on } \partial \Omega
\end{aligned}
\end{equation}
where $\Omega \subset \mathbb{R}^n$ is a bounded open set, $f> c_0$ where $c_0>0$ is a constant,
\begin{equation} \label{positivity}
D^2u-A(x, Du)>0
\end{equation}
in $\Omega$
and 
\begin{equation} \label{A_smooth}
\bar \Omega \times \mathbb{R}^n \ni (x,p) \mapsto A(x,p) \in \mathbb{R}^{n\times n}
\end{equation}
 is a $C^{\infty}$-smooth matrix valued function and $f\in C^0(\bar \Omega)$, $g\in C^0(\partial \Omega)$. 
 Note that the assumed regularity for $A,f$ and $g$ is not minimal from the point of view of mathematical analysis but still quite low as well as challenging and interesting for a first approach with numerical analysis. In the next section we will present the most important examples of cost functions 
 and derive in these special cases some properties for $A$ which result from these special cases.
 While basically working with a general $A$ in our paper we will for the sake of simplicity assume that $A$ satisfies these latter assumptions, see Section \ref{section2}. 

Note that \eqref{1} reduces to the classical  Monge-Amp\`{e}re equation when $A=0$.

As a survey and without claiming completeness we give the following list of references concerning approximation schemes for the Monge-Amp\`{e}re equation \cite{1,2,3,4,5,6,7,8,9,11,12,13}. We are not aware of any works about the finite element approximation error analysis for the Monge-Amp\`{e}re type equation (\ref{1}) with $A\neq 0$.

The first purpose of our paper is to adapt the two-scale method definition from \cite{NochettoNtogkasZhang2018} to 
a modified  two-scale method for an approximation of the Monge-Amp\`{e}re type equation
\eqref{1}. This is not completely straightforward and unique and we make an appropriate choice for the definition.
Second and mainly we show convergence of the discrete solutions defined by our two-scale method to the solution
of \eqref{1} when the two discrete parameters go to zero as well as their quotients satisfy certain bounds. For it we make certain (regularity) assumptions for the solution
of \eqref{1}, cf. Remark \ref{regularity_assumption}. 
Basically the convergence proof is achieved following the strategy from \cite{NochettoNtogkasZhang2018} by using suitable barriers and comparison principles. 
The main advance and crucial difference of our paper from \cite{NochettoNtogkasZhang2018} is that  we design completely new and much more complicated barrier functions. Apart from the barriers themselves the arguments are much more involved  since we have to handle the terms arising from $A$. This becomes especially obvious from the fact that we have only a so called $\mod O(h)$ uniqueness in the comparison principle on the discrete level which is still non-trivial. Furthermore, our convergence result for the convergence of the discrete solutions to the solution of the original problem
 is different since we require certain regularity assumptions for the solution (Remark \ref{regularity_assumption}). We are aware that this assumption is strong from the view point of viscosity solutions and a kind of artificial. We are also aware that there are basically a large variety of numerical methods 
 which are candidates which can be tested for our equation. This includes also methods
 working with regularization methods in order to achieve better regularity.  The perspective and challenge of our paper is to test the beautiful theory developed in \cite{NochettoNtogkasZhang2018} 
 in this more general case and to see how far that is possible. Hereby we include, of course, as main tool the barriers modeled on the exponential function instead of quadratic functions as in \cite{NochettoNtogkasZhang2018}. The exponential function as helping function is a very common tool in maths in general but here it should be analyzed what can be achieved by leaving the quite restricted class of quadratic functions. 
 
 Our paper is organized as follows. In section \ref{section2}
 we introduce the assumptions on the cost function and discuss the two most important cases.
 In section \ref{section3} we define our two-scale method. In section \ref{section4}
 we derive a discrete comparison principle and uniqueness for the discrete equation $\mod O(h)$.
 In section \ref{section5} we show existence of a discrete solution.
 In sections \ref{section6} and \ref{section7} we present some auxiliary facts. 
 In section \ref{section8} we study the convergence properties of the discrete model.

\section{Assumptions on the cost function and the setting in general}\label{section2}
Here we first recall the setting from \cite{PhillippisFigalli2013} and \cite{FigalliKimMcCann2013} concerning the general setup for the optimal transport problem. Then we specify the cost function and derive further properties for the matrix function $A$ in these specific cases. These motivate further assumptions for the matrix function $A$ (in addition to those from the above mentioned and below described general setup) which we will assume throughout the paper.

The general setting in  \cite{PhillippisFigalli2013} and \cite{FigalliKimMcCann2013} is motivated from the applications and the purpose to achieve 
certain regularity properties. We will present these assumptions in the following and will afterwards discuss the two most important examples of cost functions, especially it turns out that the corresponding matrices $A$ for these examples are smooth.
Let $X\subset \mathbb{R}^n$ be an open set and $u: X\mapsto \mathbb{R}$
 be a $c$-convex function, i.e., $u$ can be written as
 \begin{equation}
 u(x) = \max_{y\in \bar Y}\{-c(x,y)+\lambda_y\}
\end{equation}
for some open set $Y\subset \mathbb{R}^n$ and $\lambda_y\in \mathbb{R}$
for all $y\in \bar Y$. We are going to assume that $u$ is an Alexandrov solution
of (\ref{1}) inside some open set $\Omega \subset X$, i.e.,
\begin{equation}
|\partial^cu(E)| = \int_E f \quad \forall E\subset \Omega \quad \text{Borel},
\end{equation}
where 
\begin{equation}
\partial^cu(E) := \bigcup_{x\in E}\partial^cu(x), \quad \partial^cu(x):= 
\{y\in \bar Y:u(x)=-c(x,y)+\lambda_y\}
\end{equation}
and $|F|$ denotes the Lebesgue measure of a set $F$. For $y\in \bar Y$ we define the contact set 
\begin{equation}
\Lambda_y:= \{x\in X: u(x) = -c(x,y)+\lambda_y\}.
\end{equation}
Let $O\subset \subset Y$ be an open neighborhood of $\partial^cu(\Omega)$.
We define 
\begin{equation}
|||c|||:= \|c\|_{C^3(\bar \Omega \times \bar O)}+\|D_{xxyy}c\|_{L^{\infty}(\bar \Omega\times \bar O)}
\end{equation}
and assume
\begin{enumerate}
\item  $|||c|||< \infty $ 
\item  For every $x\in \Omega$ and $p:= -D_xc(x,y)$ with $y\in O$  it holds that 
\begin{equation}
D_{p_lp_k}A_{ij}(x,p)\xi_i\xi_j\eta_k\eta_l\ge 0 \quad \forall \xi, \eta \in \mathbb{R}^n, \quad \xi \cdot \eta =0
\end{equation}
where $A$ is defined through $c$ by
\begin{equation}
A_{ij}(x,p):= -D_{x_ix_j}c(x,y).
\end{equation}
\item 
For 
every $(x,y)\in \Omega \times O$ the maps $x \in \Omega\mapsto -d_yc(x,y)$
and $y \in O \mapsto -D_xc(x,y)$ are diffeomorphisms on their respective ranges.
\end{enumerate}

Special choices for the cost function arise from the applications, for a motivation of such choices in an economical context we refer to \cite{FigalliKimYoung-HeonMcCann2011}.
Nevertheless, the two most relevant special cases for the cost function $c$ are the following functions $c=c_1$ and $c=c_2$, cf. 
\cite{FigalliKimMcCann2013}, for which we will derive
the mapping $A$ explicitly, namely
\begin{equation}
c_1(x,y)=\frac{1}{2}|x-y|^2 \quad \text{and} \quad c_2(x,y)= -\log |x-y|.
\end{equation}
For $c=c_1$ we have
\begin{equation}
D_{x}c= x-y, \quad p=y-x
\end{equation}
and hence
\begin{equation}
A_{ij}(x,y-x)=-I,
\end{equation}
or, equivalently,
\begin{equation}
A_{ij}(x, \xi)=-I\quad \forall \xi.
\end{equation}
For $c=c_2$ we have
\begin{equation} 
\begin{aligned}
D_xc=&  -\frac{x-y}{|x-y|^2}, \\
D_{x_ix_j}c=& 2\frac{(x_i-y_i)(x_j-y_j)}{|x-y|^4}-\frac{\delta_{ij}}{|x-y|^2} \\
p=& -D_xc = \frac{x-y}{|x-y|^2}
\end{aligned}
\end{equation}
and hence
\begin{equation}
\begin{aligned}
A_{ij}\left(x, \frac{x-y}{|x-y|^2}\right) =& \frac{\delta_{ij}}{|x-y|^2}-2\frac{(x_i-y_i)(x_j-y_j)}{|x-y|^4},
\end{aligned}
\end{equation}
or, equivalently, 
\begin{equation}
A_{ij}(x, \xi) = |\xi|^2\delta_{ij}-2\xi_i\xi_j \quad \forall \xi.
\end{equation}
In these special cases the following assumption is valid.

\begin{assumption}\label{ass1}
$A$ 
is $C^{\infty}$-smooth and in addition there holds 
\begin{equation} \label{general_assumptions}
A(x,0) =0 \quad \forall x \in \bar \Omega \quad \text{or} \quad A=-I.
\end{equation}
\end{assumption}
Motivated by these two special cases and since we need such properties for technical reasons we will assume throughout the paper that 
Assumption \ref{ass1} holds.


\section{Definition of the two-scale method for Monge-Amp\`{e}re type equations} \label{section3}
 In this section we adapt the definition of the two-scale method from
 \cite{LiNochetto2018} to the more general equation \eqref{1}.
 Let $T_h=\{T_1, ..., T_N\}$, $h>0$, be a shape-regular and quasi-uniform mesh consisting of closed simplices $T_i$, $i=1, ..., N$, of diameter $ch$ where here and in the following $c$ denotes a generic constant which may vary from line to line.
We furthermore denote
\begin{equation}
\Omega_h = \inte \left(\bigcup_{i=1}^{N}T_i\right),
\end{equation}  
let $N_h$ be the nodes of $T_h$ and write $N_h^b=\{x_i\in N_h: x_i \in \partial \Omega_h\}$
for the boundary nodes and $N_h^0=N_h\setminus N_h^b$ for the interior nodes. We furthermore assume that 
$\Omega$ is convex, that $N_h^b\subset \partial \Omega$ and denote the space of continuous 
functions on $\Omega_h$, which are linear on $T_i$ for every $i=1, ..., N$, by $V_h$.
We denote the set of $n\times n$ matrices of real numbers by $\mathbb{R}^{n\times n}$ and the subset of orthogonal matrices by $O(n)$, furthermore, we write elements $V \in \mathbb{R}^{n\times n}$ by $V=(v_j)_{j=1}^d$ where $v_j$ are the columns of $V$ with respect to the standard basis in $\mathbb{R}^n$.
  
We denote the unit sphere in $\mathbb{R}^n$ by $S$ and for $\theta>0$ we let $S_{\theta}$
be a finite subset of $S$ with the property that
\begin{equation}
\forall\ v \in S� \quad  \exists\ v_{\theta}\in S_{\theta} : \quad |v-v_{\theta}| \le \theta.
\end{equation}
Especially, we may assign to an element $V=(v_j)\in O(n)$ a matrix $V=(v_j^{\theta})$ where $v^{\theta}_j=(v_j)_{\theta}$ and denote the set of 
all such matrices by $O^{\theta}(n)$. Note that $O^{\theta}(n) \not \subset O(n)$ in general.

In addition to the meshsize $h$ which will serve as the fine scale in the remaining part of the paper we introduce in the following a coarse scale $\delta>h$ as a second discrete parameter which will serve as step size in difference quotients defining discrete derivatives.
For $x_i \in N_h^0$ let
\begin{equation}
\delta_i = \min\{\delta, \dist(x_i, \partial \Omega_h)\}
\end{equation}
and note that $\delta_i \ge ch$ where $c$ does not depend on $h$ and that $B(x_i, \delta_i)\subset \Omega_h$. Here, $B(x_i, \delta_i)$ denotes the open ball of radius $\delta_i$ around $x_i$.
For $w \in C^0(\overline{\Omega_h})$ we define the one-sided first order difference operator
\begin{equation}
\nabla_{\delta}w(x_i, v_j)=\frac{w(x_i+\delta_iv_j)-w(x_i)}{\delta_i}
\end{equation}
and the centered second order difference operator
\begin{equation} \label{2.1}
\nabla^2_{\delta}w(x_i; v_j)=\frac{w(x_i+\delta_iv_j)-2w(x_i)+w(x_i-\delta_iv_j)}{\delta_i^2}
\end{equation}
for $x_i\in N_h^0$ and $v_j \in S_{\theta}$. 
Here we choose one-sidedness in the definition for the first order difference operator but remark that centered differences might also work. Altogether we have three discrete parameters which we will summarize as
\begin{equation}
\varepsilon=(h, \delta, \theta)
\end{equation}
where $h, \delta, \theta>0$ and $\delta >h$. To the two latter inequalities 
we will sometimes refer to by writing $\varepsilon>0$.
For the following we will fix $\varepsilon$ for a while and will analyze the corresponding discrete model. 
Then later in a second step we will discuss the limit $\varepsilon \rightarrow 0$
and a necessary coupling between the parameters in order to achieve convergence
of the solutions of the discrete equations to the solution of the original equation.

In the following definition we generalize the two-scale operator from \cite{LiNochetto2018}.

\begin{definition}\label{defi_1}
For $x_i \in N_h^0$ we define for any $w_h\in V_h$
\begin{equation}\label{equ_1}
\begin{aligned}
T_{\varepsilon}[w_h](x_i):=&\min_{v^{\theta}\in O^{\theta}(n) }
\Big(\prod_{j=1}^d\left(\nabla_{\delta}^2w(x_i, v^{\theta}_j)-(v^{\theta}_j)^TA(x_i, \nabla_{\delta}w(x_i, e_k))v^{\theta}_j\right)^{+}\\
&-\sum_{j=1}^d\left(\nabla_{\delta}^2w(x_i, v^{\theta}_j)-(v^{\theta}_j)^TA(x_i, \nabla_{\delta}w(x_i, e_k))v^{\theta}_j\right)^{-}
\Big)
\end{aligned}
\end{equation}
where two remarks are in order concerning our notation. Firstly, we write $(\cdot)^{+}=\max(\cdot, 0)$ and $(\cdot)^{-}=-\min(\cdot, 0)$ to denote the non-negative and non-positive part of $(\cdot)$, respectively. Secondly,
we abbreviate 
\begin{equation}
\nabla_{\delta}w(x_i, e_k):=\left(\nabla_{\delta}w(x_i, e_k)\right)_{k=1}^d
\end{equation}
where $w\in V_h$, $\delta>h$, $x_i\in N_h^0$ and $(e_k)_{k=1}^d$ denotes the canonical basis in $\mathbb{R}^d$ to simplify the notation in expression (\ref{equ_1}).
\end{definition}

By using the discrete two-scale operator from Definition \ref{defi_1} we obtain the following discrete version of the
Monge-Amp\`{e}re type problem (\ref{1}). 

\begin{definition}
For a given  (triple) $\varepsilon>0$  a two-scale method solution of  (\ref{1}) is a function
 $u_{\varepsilon}\in V_h$ such that $u_{\varepsilon}(x_i)=g(x_i)$
for all $x_i \in N^b_h$ and 
\begin{equation} \label{discreteMongeAmpere}
T_{\varepsilon}[u_{\varepsilon}](x_i):= f(x_i)
\end{equation}
for all $x_i \in N^0_h$.
\end{definition}

In view of the widely used convention in numerical analysis to denote discrete solutions
with the subscript $h$, i.e. $u_h$, we will write in the following ocassionally $u_h$ instead of $u_{\varepsilon}$.

\section{Discrete $Q$-convexity, monotonicity and discrete comparison principle $\mod O(h)$}
\label{section4}
To simplify the notation we use the following conventions. Firstly, in the setting from Definition \ref{defi_1} we will abbreviate 
in the following
\begin{equation} \label{def_Q}
Q(x_i, v_j)=Q_w(x_i, v_j)=\nabla_{\delta}^2w(x_i, v_j)-(v_j)^TA(x_i, \nabla_{\delta}w(x_i, e_k))v_j
\end{equation}
so that (\ref{equ_1}) takes the form
\begin{equation}\label{equ_2}
\begin{aligned}
T_{\varepsilon}[w_h](x_i)=&\min_{v^{\theta}\in O^{\theta}(n) }
\Big(\prod_{j=1}^d\left(Q(x_i, v^{\theta}_j)\right)^{+}
-\sum_{j=1}^d\left(Q(x_i, v^{\theta}_j)\right)^{-}
\Big).
\end{aligned}
\end{equation}
Secondly, when a variable ranges in a discrete set we sometimes emphasize this fact by adding a superscript to this variable
which is linked to this discrete set, e.g. we write $v^{\theta} \in O^{\theta}(n)$ and $w_h \in V_h$
but equally $v\in O^{\theta}(n)$ and $w \in V_h$, respectively.
Throughout this section we assume that (the triple) $\varepsilon>0$ is fixed.
\begin{definition}  \label{definition_1}
We say that $w_h \in V_h$ is discretely $Q$-convex if
\begin{equation} \label{equ_2_}
Q(x_i; v_j) \ge 0 \quad \forall x_i \in N^0_h, \quad \forall v_j \in O^{\theta}(n).
\end{equation}
\end{definition}
Note, that discrete $Q$-convexity of $w_h$ does not imply convexity of $w_h$ in general.
\begin{lemma} \label{lemma_1}
If $w_h \in V_h$ satisfies 
\begin{equation}
T_{\varepsilon}[w_h](x_i)\ge 0 \quad \forall x_i \in N_h^0,
\end{equation}
then $w_h$ is discretely $Q$-convex and as a consequence
\begin{equation}
T_{\varepsilon}[w_h](x_i) = \min_{v\in O^{\theta}(n)}\prod_{j=1}^dQ(x_i, v_j).
\end{equation}
\end{lemma}
\begin{proof}
We distinguish two cases depending on whether
$T_{\varepsilon}[w_h](x_i)>0$ or not. Let $v=(v_j)_{j=1}^d\in O^{\theta}(n)$ be a $d$-tuple that realizes
the minimum in the definition of $T_{\varepsilon}[w_h](x_i)$ and note that
\begin{equation}
\prod_{j=1}^dQ(x_i; v_j)^{+}\ge 0, \quad \sum_{j=1}^dQ(x_i; v_j)^{-}\ge 0.
\end{equation}

(i) Assume that $T_{\varepsilon}[w_h](x_i)>0$. 
The expression $T_{\varepsilon}[w_h](x_i)$ is defined as a product, cf. (\ref{equ_2}), so that its positivity implies
the positivity of all its factors. These positive factors are differences of type $a-b$ of non-negative numbers $a$ and $b$ so that we also always necessarily have $a>0$. This implies that each quantity $Q(x_i; v_j)^{+}$ is also positive and hence the sum-term in (\ref{equ_2}) vanishes.

(ii) Assume that $T_{\varepsilon}[w_h](x_i)=0$. Using again the representation from (\ref{equ_2})
of this expression as a difference of a product and a sum we make the following conclusion. 
If this product is positive then by (i) the sum vanishes and hence the product and the sum vanish so that $Q(x_i; v_j)= 0$ and the claim follows as well.
\end{proof}

Note that Lemma \ref{lemma_1} and Definition \ref{definition_1} make formally sense when in (\ref{equ_1}) and (\ref{equ_2}) the superscript $\theta$ is omitted
 and Lemma  \ref{lemma_1} is then even also true.

We need a definition in which we introduce a family of subspaces of $V_h$.
\begin{definition} \label{Definition_V_h}
For $\Lambda, h>0$ we define
\begin{equation} \label{def_V_h}
\begin{aligned}
V_h^{\Lambda} = \left\{
v_h \in V_h: \exists \eta \in C^{\infty}(\bar \Omega) , I_h\eta = v_h,
\|\eta\|_{C^2(\bar \Omega)}\le \Lambda\right\}
\end{aligned}
\end{equation} 
where $I_h$ denotes the usual Lagrange interpolation operator.
\end{definition}
In the course of the paper it will turn out that when fixing a sufficiently large $\Lambda>0$
all considerations can be done (and will be done) for the sequence of discrete spaces $(V_h^{\Lambda})_{h>0}$.
Here, $\Lambda$ will be chosen depending on the data of the problem, i.e. 
depending on $A$, $g$, $f$, $\Omega$ and the uniform (with respect to $h$)
parameters of the triangulation. Interestingly, 
we only bound derivatives up to the {\it second} order in the definition of $V_h^{\Lambda}$. So arguments solely based on interpolation do not work hence they require at least bounds for the third derivative of the interpolating function.

\begin{remark} \label{remark_1}
For the sake of a simplier notation we will write in the following again $V_h$ instead of 
$V_h^{\Lambda}$ with $\Lambda$ large and fixed. We will comment on $\Lambda$ where necessary.
\end{remark}

As a consequence of Remark \ref{remark_1} we have the following discrete version of the 
fact that the derivative of a differentiable function vanishes in interior extremal points. Let $v_h\in V_h$, $v_j\in S_{\theta}$
and $z\in N_h^0$ be a maximum (or a minimum) of $v_h$ then
\begin{equation}\label{rolle}
\nabla_{\delta}v_h(z, v_j) = O(h).
\end{equation}

In the next lemma we show that $T_{\varepsilon}$ is monotone$\mod\ O(h)$. 

\begin{lemma} \label{lemma10}
Let $u_h, w_h\in V_h$ be discretely $Q$-convex. If $u_h-w_h$
attains a maximum at an interior node $z \in N_h^0$ then
\begin{equation}
T_{\varepsilon}[w_h]\ge T_{\varepsilon}[u_h]+O(h)
\end{equation}
in $\Omega_h$.  Here, the constant hidden in the $O(h)$-notation depends on $\Lambda$.
\end{lemma}
\begin{proof}
If $u_h-w_h$ attains a maximum at $z \in N_h^0$ then
\begin{equation} \label{equ1}
u_h(z)-w_h(z) \ge u_h(x_i)-w_h(x_i) \quad \forall x_i \in N_h.
\end{equation}
Since $u_h$ and $w_h$ are piecewise linear this inequality can be generalized to
\begin{equation}
u_h(z)-w_h(z) \ge u_h(x)-w_h(x) \quad \forall x \in \Omega_h.
\end{equation}
Especially, evaluating this for difference quotients gives in view of  
(\ref{2.1}) that
\begin{equation} \label{equ10}
\nabla_{\delta}^2u_h(z, v_j)\le \nabla_{\delta}^2w_h(z, v_j) \quad \forall v_j \in S_{\Theta}.
\end{equation}
It remains to show that
\begin{equation}
Q_{u_h}(z, v_j) \le Q_{w_h}(z, v_j)+O(h)
\end{equation}
which can be reduced by (\ref{equ10}) to
\begin{equation} \label{equ11}
(v_j)^TA(z, \nabla_{\delta}w_h(z, e_k))v_j \le (v_j)^TA(z, \nabla_{\delta}u_h(z, e_k))v_j+O(h).
\end{equation}
But this follows since
\begin{equation}
\nabla_{\delta}w_h(z, e_k) -\nabla_{\delta}u_h(z, e_k)=O(h), 
\end{equation}
cf. (\ref{rolle}),
and
\begin{equation}
|\nabla_{\delta}w_h(x_i, e_k)|\le C,
\end{equation}
cf. Definition \ref{Definition_V_h},
from the continuity of $A$.
\end{proof}

We will use the following notation.
\begin{remark}
Given two functions $f_1=f_1(x,h)$ and $f_2=f_2(x,h)$ where $x\in S$
ranges in a certain parameter set $S$ as well as the discretization parameter 
$h>0$ we write 
\begin{equation}
f_1 \le f_2 \text{ for all } x\in S \mod O(h)
\end{equation}
if there exists a constant $C>0$ which does not depend on $h$, $f_1$ or $f_2$ such that
\begin{equation}
f_1(x,h) \le f_2(x,h) + Ch \text{ for all }x\in S  \text{ and all } h>0.
\end{equation}
When the parameter set $S$ is clear from the context we will not mention it explicitly.
\end{remark}
We show the following discrete comparison principle, cf. Lemma \ref{comparison}. Note that the inequality in the lemma includes the error term $O(h)$. This linear error order is not trivial in the context of a nonlinear second order operator and the use of first order finite elements for the following  reason. If one derives the inequality in the following lemma firstly via a well-known comparison principle on the continuous level and later on transfers this by using interpolation estimates to the discrete level then usually third derivatives appear. But according to Definition \ref{Definition_V_h} third derivatives of the interpolating functions may be arbitrary large and hence there appears an error term which is of the size of the product of the third derivative of an artificial interpolating function and $h$ and hence possibly large and especially larger than $O(h)$.

\begin{lemma} \label{comparison}
Let $u_h, w_h \in V_h$ with $u_h\le w_h$ on the boundary $\partial \Omega_h$ be such that
\begin{equation}
T_{\varepsilon}[u_h](x_i) \ge T_{\varepsilon}[w_h](x_i) > 0 \quad \forall x_i \in N_h^0. \label{equ12}
\end{equation}
Then we have $u_h \le w_h$ in $\Omega_h$$\mod O(h)$.
\end{lemma}
\begin{proof}
Since $u_h, w_h \in V_h$, it suffices to prove $u_h(x_i)\le w_h(x_i)$ for all $x_i \in N_h^0$.
In view of Lemma \ref{lemma_1} we may write inequality (\ref{equ12}) as
\begin{equation} \label{equ13}
\min_{v\in O^{\theta}(n)}\prod_{j=1}^dQ_{u_h}(x_i, v_j) \ge  \min_{v\in O^{\theta}(n)}\prod_{j=1}^dQ_{w_h}(x_i, v_j)> 0 \quad  \forall x_i \in N_h^0.
\end{equation}
Now we distinguish cases. For it we fix  constants $C_1, C_2>0$ which depend only on the data of the problem, i.e. on $A$, $f$, $\Omega$, and which will be specified later.

(i) Let us assume
\begin{equation} \label{equ13_}
\min_{v\in O^{\theta}(n)}\prod_{j=1}^dQ_{u_h}(x_i, v_j) -C_1 h> \min_{v\in O^{\theta}(n)}\prod_{j=1}^dQ_{w_h}(x_i, v_j) \quad  \forall x_i \in N_h^0.
\end{equation}
 We argue by contradiction and assume that there is $x_k \in N_h^0$ such that
\begin{equation} \label{equ20}
u_h(x_k)-w_h(x_k)=\max_{x_i\in N_h^0}u_h(x_i)-w_h(x_i)> 0.
\end{equation}
Similarly, as in Lemma \ref{lemma10} we conclude that
\begin{equation}
Q_{u_h}(x_k, v_j) \le Q_{w_h}(x_k, v_j)+C_3 h \quad \forall v_j \in S_{\theta}
\end{equation}
where $C_3>0$ is a suitable constant which depends only on the data of the problem (and especially not on $h$).
Taking the product on both sides and after this the infimum on the left-hand side of the equation
yields
\begin{equation}
\min_{v\in O^{\theta}(n)}\prod_{j=1}^dQ_{u_h}(x_k, v_j)\le \prod_{j=1}^dQ_{w_h}(x_k, \tilde v_j) +C_2h \quad \forall \tilde v \in O^{\theta}(n)
\end{equation}
where $C_2$ is a suitable constant which depends only on $C_3$ and the data of the problem.
W.l.o.g. we may also take the infimum over all $\tilde v\in O^{\theta}(n)$ on the right-hand side of the inequality.

Combining this with nequality  
(\ref{equ13_}) we obtain a contradiction provided $C_1$ is sufficiently large compared to $C_2$. This finishes case (i).

(ii) Let us assume the other case, i.e. we have
\begin{equation} \label{equ14_}
\min_{v\in O^{\theta}(n)}\prod_{j=1}^dQ_{u_h}(x_i, v_j) -C_1 h  \le  \min_{v\in O^{\theta}(n)}\prod_{j=1}^dQ_{w_h}(x_i, v_j) \quad  \forall x_i \in N_h^0
\end{equation} 
where now $C_1$ is fixed as it turned out to be necessary  in case (i). 

The strategy of the proof will now be as follows.
We show that there are constants $h_0, C_4, C_5>0$ and an auxiliary function $q_h \in V_h$ which depend on the data of the problem such that 
\begin{equation} \label{consequence_barrier}
\begin{aligned}
T_{\varepsilon}[u_h+\alpha q_h](x) > T_{\varepsilon}[w_h](x)+C_1h \\
\forall 0<h<h_0 \quad \quad \forall \alpha \ge C_4h \text{ sufficiently small}
\end{aligned}
\end{equation}
and
\begin{equation}
\|q_h\|_{L^{\infty}(\Omega)}\le C_5.
\end{equation}
From this we conclude by using (i) that
\begin{equation}
u_h \le w_h + C_4C_5h
\end{equation}
and hence the claim.

We choose  $\tilde x \in \mathbb{R}^n$ such that $\dist(\tilde x, \bar \Omega)\ge 1$,
 $\lambda>0$ large and define the
strictly convex function
\begin{equation} \label{ansatz_function1}
q(x) = e^{\lambda |x-\tilde x|^2}-R
\end{equation}
where $R=R(\lambda, \Omega)$ is so that $q\le 0$ in $\Omega$ and especially in $\bar \Omega_h$.
We now define 
\begin{equation}
q_h=I_hq,
\end{equation}
perform some relevant calculations on the level of $q$ instead of $q_h$ and translate them to $q_h$ afterwards by using the standard interpolation estimate
\begin{equation} \label{interpolation}
\|q_h-q\|_{C^m(\Omega_h)} \le c_{m,r} h^r \|q\|_{C^{m+r}(\bar \Omega)}, \quad m,r \in \mathbb{N},
\end{equation}
where $c_{m,r}$ are suitable constants.
We have
\begin{equation}
\begin{aligned}
D_iq(x) =& 2\lambda (x_i-\tilde x_i) e^{\lambda |x-\tilde x|^2} \\
D_iD_jq(x) =& 2\lambda e^{\lambda|x-\tilde x|^2}\delta_{ij}+4 \lambda^2e^{\lambda |x-\tilde x|^2}(x_i-\tilde x_i)(x_j-\tilde x_j).
\end{aligned}
\end{equation}
For $x \in N_h^0$ and $v \in O^{\theta}(n)$ we calculate
\begin{equation} \label{representation_Q}
\begin{aligned}
Q_{u_h+\alpha q}(x, v_r) =& 
\nabla_{\delta}^2u_h(x, v_r)+\alpha \nabla_{\delta}^2q(x, v_r) \\
& - v_r^TA(x, \nabla_{\delta}u_h(x, e_k)+\alpha \nabla_{\delta}q(x, e_k))v_r \\
=& \nabla_{\delta}^2u_h(x, v_r)+\alpha\left(O(\delta)+D_{v_r}D_{v_r}q(x)\right)\\
&-v_r^TA(x, \nabla_{\delta}u_h(x, e_k)+\alpha D_kq(x)+\alpha O(\delta))v_r \\
=& \nabla_{\delta}^2u_h(x, v_r)+\alpha O(\delta) + 2 \alpha \lambda e^{\lambda |x-\tilde x|^2}\delta_{ij}v_{ri}v_{rj} \\
&+4 \alpha \lambda^2e^{\lambda |x-\tilde x|^2}(x_i-\tilde x_i)(x_j-\tilde x_j)v_{ri}v_{rj} \\
&-v_r^TA(x, \nabla_{\delta}u_h(x, e_k)+ 2 \alpha\lambda (x_k-\tilde x_k)e^{\lambda|x-\tilde x|^2}+\alpha O(\delta))v_r
\end{aligned}
\end{equation}
Here, the constant hidden in the $O(\delta)$-notation depends on $q$, more precisely, on its higher order derivatives.
Since $T_{\varepsilon}[u_h]=:f>0$ there is $\alpha_0>0$ such that
\begin{equation} \label{infimum_attained}
T_{\varepsilon}[u_h+\alpha q_h]>0
\end{equation}
for all $\alpha \in (0, \alpha_0)$. Hence for these $\alpha$ we have
\begin{equation}
T_{\varepsilon}[u_h+\alpha q_h](x) = \min_{v \in O^{\theta}(d)}\prod_{j=1}^dQ_{u_h+\alpha q_h}(x, v_j)
\end{equation}
and therefore by abbreviating $Q(\alpha, j)=Q_{u_h+\alpha q_h}(x, v_j)$ (where the arguments $x$ and $v_j$ are assumed to be implicitly clear from the context)
for $\alpha \in (0, \alpha_0)$ and $j\in \{1, ..., d\}$ we write
\begin{equation} \label{deriv_of_prod}
\begin{aligned}
\frac{d}{d\alpha}T_{\varepsilon}[u_h&+\alpha q_h](x)_{|\alpha=0} \\
=& \sum_{j=1}^dQ(0, 1)...
Q(0, j-1)\frac{d}{d\alpha}Q(\alpha, j)_{|\alpha=0}Q(0, j+1)...Q(0, d).
\end{aligned}
\end{equation}
Note that the arguments $x$ and $v$ which appear here implicitly on the right-hand side are chosen 
obviously - namely as on the left-hand side of the equation as far as $x$ is concerned; the matrix $v$ is chosen so that in the point $x$ the infimum is attained in the definition (\ref{infimum_attained}).
In order to evaluate (\ref{deriv_of_prod}) we first calculate the derivative of the expression in (\ref{representation_Q}). Observe that there holds for all $k \in \{1, ..., d\}$ that
\begin{equation}
\begin{aligned}
\frac{d}{d\alpha}Q(\alpha, k)
=& O(\delta)+2 \lambda e^{\lambda |x-\tilde x|^2}\delta_{ij}v_{ik}v_{jk} \\
&+ 4 \lambda^2e^{\lambda |x-\tilde x|^2}(x_i-\tilde x_i)(x_j-\tilde x_j)v_{ki}v_{kj} 
\\
&-v_k^T \left(\frac{\partial A}{\partial p_l}(x, \nabla_{\delta}u_h(x, e_k))2 \lambda(x_l-\tilde x_l)e^{\lambda |x-\tilde x|^2}+O(\delta)\right)v_k.
\end{aligned}
\end{equation}
Assuming that $\theta$ is sufficiently small we have
\begin{equation}
(x_i-\tilde x_i)(x_j-\tilde x_j)v_{ki}v_{kj} \ge \frac{1}{2}|x-\tilde x|^2
\end{equation}
for all $k \in \{1, ..., d\} \setminus \{k_0\} $ and all $v \in O^{\theta}(d)$
where $k_0=k_0(v)\in \{1, ..., d\}$ is suitable.
Now having $\Lambda$ fixed in the definition of $V^{\Lambda}_h$ and choosing $0<h_0\le 1$ (at the moment not further specified)  we may assume that
\begin{equation}
\frac{d}{d\alpha}Q(\alpha, k)_{|\alpha=0} \ge \lambda^2e^{\lambda |x-\tilde x|^2}|x-\tilde x|^2>0
\end{equation}
provided $\lambda>0$ is sufficiently large.
Furthermore, for $\alpha_0=\alpha_0(\lambda)>0$ sufficiently small the quantities $Q( \alpha, k)$ are uniformly 
with respect to $k$ and with respect to $\alpha \in (-\alpha_0, \alpha_0)$ bounded by a positive constant from below.
Hence we arrive at
\begin{equation}
\frac{d}{d\alpha}T_{\varepsilon}[u_h+\alpha q_h](x)_{|\alpha=0}\ge \mu_0>0
\end{equation}
with a suitable fixed $\mu_0>0$.
An expansion of $T_{\varepsilon}[u_h+\alpha q_h](x)$ around 0 yields the existence of $\alpha \in (0, \alpha_0)$
such that
\begin{equation}
T_{\varepsilon}[u_h+\alpha q_h](x)>T_{\varepsilon}[u_h](x)+ \frac{\alpha}{2}\mu_0.
\end{equation} 
Clearly, by assuming that $h_0$ is sufficiently small the above construction shows that we can realize property
(\ref{consequence_barrier}) with this specific $\alpha$. Actually, we have in addition to choose $C_4>0$
but as long it is not too large it does not matter how we choose it. This finishes the proof.
\end{proof}

\section{Existence of discrete solutions} \label{section5}
We now prove uniqueness$\mod\ O(h)$ and existence  
of a discrete solution $u_{\varepsilon}\in V_h$ of (\ref{discreteMongeAmpere}).
Here, the uniqueness $\mod O(h)$ means, that given two discrete solutions 
$u^1_{\varepsilon}, u^2_{\varepsilon}$ of (\ref{discreteMongeAmpere}) there holds $u^{i}_{\varepsilon}\le u^{j}_{\varepsilon}$ $\mod O(h)$ for all $i,j\in \{1,2\}$.

\begin{lemma}
There exists $u_{\varepsilon}\in V_h$ which satisfies the discrete Monge-Amp\`{e}re type equation (\ref{discreteMongeAmpere}) and which is unique$\mod O(h)$. Furthermore, $\|u_{\varepsilon}\|_{L^{\infty}(\Omega)}$ does not depend on
the parameter $\varepsilon=(h, \delta, \theta)$.
\end{lemma}
\begin{proof}
Let us fix $\varepsilon>0$.
The uniqueness$\mod O(h)$ of a solution of (\ref{discreteMongeAmpere}) follows from Lemma \ref{comparison}.
Hence it remains to show existence. For it we construct 
a special monotone sequence of discretely $Q$-convex subsolutions $\{u_{\varepsilon}^k\}_{k=0}^{\infty}$ of (\ref{discreteMongeAmpere}) from which we will select a subsequence which converges to the desired discrete solution of 
(\ref{discreteMongeAmpere}). The construction is by induction and works as follows.

(i) {\it Claim: There is $u_h^0 \in V_h$ such that $u_h^0=I_hg$ on $\partial \Omega_h$ and 
\begin{equation}
T_{\varepsilon}[u_h^0](x_i)\ge f(x_i) \quad \forall x_i \in N_h^0.
\end{equation}
}  

{\it Proof of the claim:}

(a) We give a short proof in the case that there is $C^{2,\alpha}$-regularity of the solution available. Let us assume that there is $0<\alpha<1$ such that the problem (\ref{1}) with right-hand side $f$ replaced by $f+1$ has a solution $u \in C^{2,\alpha}(\bar \Omega)$. Setting $u_h^0=I_hu$ and using the interpolation estimates from \cite[Lemma 4.1]{NochettoNtogkasZhang2018} as well as the continuity of $A$ yields the claim.

(b) In the general case we proceed without using (regularity of) the solution of (\ref{1}).
Let $q$ denote the auxiliary function from \eqref{ansatz_function1} with arbitrary choice of $R$, e.g. $R=0$. Let $w$
be a smooth function in a ball $B_{L}(0)$, $L>0$ large, let us say $2\bar \Omega \subset B_L(0)$, with
\begin{equation}
w(x_i)=g(x_i)-q(x_i), \quad x_i \in N_h^b.
\end{equation}
Such a function can easily be obtained by fixing it in $N_h^b$ and then extending it as smooth function to $B_L(0)$. 
But we would like to have that the size of $|Dw(x)|$ and $|D^2w(x)|$ is of order $ O(\lambda e^{\lambda |x-\tilde x|^2})$
and hence small compared to the order $O(\lambda^2 e^{\lambda |x-\tilde x|^2})$ which is the size of $D^2q(x)$. For it we define
an artificial domain $\tilde \Omega \subset \mathbb{R}^d$ with smooth boundary $\partial \tilde \Omega$ passing through all elements 
of $N_h^b$, i.e. $N_h^b \subset \partial \tilde \Omega$.
We extend $g-q$ from $N_h^b$ to a function $b\in C^{1, \alpha}(\partial \tilde \Omega)$ with
\begin{equation} \label{representation}
\|b\|_{C^{1, \alpha}_x(\partial \tilde \Omega)} \le \mu_x\lambda^2 e^{\lambda |x-\tilde x|^2}
\end{equation}
for all $x \in \partial \tilde \Omega$ where we may and will choose here the constant 
\begin{equation}
0<\mu_x<\mu_0
\end{equation}
with $\mu_0>0$ small. Here, we denote
\begin{equation}
\|b\|_{C^{1, \alpha}_x(\partial \tilde \Omega)} = |b(x)| + \sum_{i=1}^d\|D_ib\|_{C^{0,\alpha}_x(\partial \tilde \Omega)}
\end{equation}
where
\begin{equation}
\|D_ib\|_{C^{0,\alpha}_x(\partial \tilde \Omega)} = |D_ib(x)|+\sup_{y \in \partial \tilde \Omega, y \neq x}\frac{|D_ib(x)-D_ib(y)|}{|x-y|^{\alpha}}, \quad x \in \partial \tilde \Omega.
\end{equation}
To give derivatives (and their norms) of a function being defined on the hypersurface  $\partial \tilde \Omega$ a sense
we either consider these with respect to a fixed finite selection of local coordinate systems covering $\partial \tilde \Omega$
or with respect to an arbitrary but fixed extension to an open neighborhood of $\partial \tilde \Omega$ of the corresponding functions.
In order to understand how the representation (\ref{representation}) is  possible, we explain this for the most non-trivial 
case, i.e. on the level of the H\"older norm of the derivative. Given a small choice for $\mu_x>0$, we  estimate
for $i\in \{1, ..., d\}$ and $x,y \in \partial \tilde \Omega$ that
\begin{equation}
\begin{aligned}
\frac{|D_ib(x)-D_ib(y)|}{|x-y|^{\alpha}} =& \frac{|D_ib(x)-D_ib(y)|}{|x-y|^{\alpha}|x-y|^{1-\alpha}}|x-y|^{1-\alpha} \\
\approx& D^2 q(x) |x-y|^{1-\alpha} \\
\le& D^2q(x) \mu_x^{1-\alpha}
\end{aligned}
\end{equation}
if $|x-y|\le \mu_x$. In the other case, i.e. when $|x-y|>\mu_x$, we estimate
\begin{equation}
\begin{aligned}
\frac{|D_ib(x)-D_ib(y)|}{|x-y|^{\alpha}} \le&
\frac{|D_ib(x)|+|D_ib(y)|}{\mu_x^{\alpha}} \\
\le& O\left(\frac{\lambda}{\mu_x}e^{\lambda |x-\tilde x|^2}\right).
\end{aligned}
\end{equation}
Then we solve the Dirichlet problem
\begin{equation}
\begin{aligned}
\Delta w =&0 \quad \text{in } \tilde \Omega \\
w =&b \quad \text{on } \partial \tilde \Omega
\end{aligned}
\end{equation}
and obtain by classical PDE-theory a solution $w \in C^{3, \alpha}\left(\overline{\tilde \Omega}\right)$ which satisfies
the Schauder-estimate
\begin{equation} \label{schauder}
\|w\|_{C^{3,\alpha}(\overline{\tilde \Omega})} \le c \left(
\|w\|_{C^0\left(\overline{\tilde \Omega}\right)}+\|b\|_{C^{1, \alpha}(\partial \tilde \Omega)}\right).
\end{equation}
Noting that 
\begin{equation}
\|w\|_{C^0_x(\tilde \Omega)} = O(q(x))
\end{equation}
we conclude that $w$ satisfies the desired properties.
Now we set
\begin{equation}
u^0=w+q
\end{equation}
and then 
\begin{equation}
u^0_{\varepsilon}:= I_hu^0. 
\end{equation}
By construction $u^0_{\varepsilon}$ has the correct boundary values, i.e. 
$u^0_{\varepsilon}(x_i)=g(x_i)$ when $x_i \in N_h^b$ and it satisfies
\begin{equation}
T_{\varepsilon}[u^0_{\varepsilon}](x_i) \ge f(x_i)
\end{equation}
for all $x_i \in N_h^0$  provided $\lambda$ is large in view of the interpolation error estimates in \cite[Lemma 4.1]{NochettoNtogkasZhang2018}. Hence we have constructed $u_{\varepsilon}^0$ as desired. Note that we proved here a little bit more than needed. In order to apply \cite[Lemma 4.1]{NochettoNtogkasZhang2018} it suffices to have only the Schauder estimate (\ref{schauder}) on the level of $C^{2,\alpha}$ available.
Furthermore, we remark that the construction can be done so that the $L^{\infty}$-norm of $u^0_{\varepsilon}$ can be estimated uniformly in $h$.

(ii) We follow a Perron construction from \cite{NochettoNtogkasZhang2018} and use induction. 
First we label all interior nodes, let us say, $N_h^0=\{x_1, ..., x_m\}$, $m\in \mathbb{N}$. 
The induction begins with
$u_h^0 \in V_h$ from (i).
Let us assume we already have constructed $u_h^k\in V_h$ for some $k\in \mathbb{N}$ such that
\begin{equation} \label{properties}
\begin{aligned}
u_h^k \ge& u_h^0 \\
u_h^k(x_i) =& I_hg(x_i), \quad x_i \in N_h^b, \\
T_{\varepsilon}[u_h^k](x_i)\ge& f(x_i), \quad x_i \in N_h^0.
\end{aligned}
\end{equation}
In order to construct $u_h^{k+1}\in V_h$ which satisfies 
\begin{equation} \label{property_additional}
u_h^{k+1}\ge u_h^k
\end{equation}
as well as the properties \eqref{properties} with $k$ replaced by $k+1$
we first define auxiliary functions $u_h^{k,i}\in V_h$, $i=0, ..., m$.
We set
\begin{equation}
u_h^{k,0}:=u_h^k.
\end{equation}
Assume that $u_h^{k, i-1}\in V_h$ is already defined, $i\ge 1$. In order to define $u_h^{k,i}\in V_h$ we increase
(only) the value of $u_h^{k,i-1}(x_i)$ (eventually) until
\begin{equation} \label{increasement}
T_{\varepsilon}[u_h^{k,i}](x_i)=f(x_i).
\end{equation}
This defines $u_h^{k,i}$. The equality in \eqref{increasement} can indeed be achieved under this process which becomes
clear when we look at Lemma \ref{lemma_1} and (\ref{def_Q}). Noting that the centered second 
differences appearing in this definition of $Q$ are decreasing with slope $\frac{c}{h^2}$, $c$ a generic constant,  with respect to the central 
value for all directions and that all other expressions therein change under this process at most by a rate of $\frac{c}{h}$ the equality in \eqref{increasement} can clearly be achieved for $h$ sufficiently small.
This process potentially increases the second centered differences at
all the other nodes $x_j$, $j \neq i$ at a rate $\frac{c}{h^2}$ and changes lower order terms at most at a rate $\frac{c}{h}$.
Hence
\begin{equation}
 T_{\varepsilon}[u_h^{k,i}](x_j)\ge T_{\varepsilon}[u_h^{k,i-1}](x_j) \ge f(x_j) \quad \forall j \neq i.
 \end{equation}
 We repeat this process with the remaining nodes $x_j$ for $i<j\le m$ and set
 \begin{equation}
 u_h^{k+1}:=u_h^{k,m}.
 \end{equation}
 Note that the 'sufficient smallness' of $h$ can be chosen here uniformly.
 Clearly, $u_h^{k+1}$ satisfies \eqref{properties} and \eqref{property_additional}.
 
(iii)  We derive an a priori $L^{\infty}$-bound for the sequence $(u_h^k)_{k\in \mathbb{N}}$. 
The lower bound for this sequence follows from the remarks at the end of steps (i) and (ii). 
Recall that by (\ref{general_assumptions}) we have $A(x, \cdot)=0$ or $A=-I$.
The upper bound is chosen as follows.
We set $\tilde b_h=\max_{x_i \in N_h^b}g(x_i) \in V_h$. In the case $A(x, \cdot)=0$ we set
$b_h=\tilde b_h$ and in the case $A=-I$ we set
$b_h=\tilde b_h +c(\Omega)- (1- \frac{1}{4} \min f)I_h|x|^2 $ where $c(\Omega)$ is a positive constant
which depends on $\Omega$.
Clearly, by the comparison principle $b_h$ is an upper barrier$\mod O(h)$ for the sequence $(u_h^k)_{k\in \mathbb{N}}$ and we are finished, note that we assume here that $h$ is small.

(iv) Since $(u_h^k(x_i))_{k=1}^{\infty}$
is monotone and bounded from above for all $x_i \in N_h^0$ it converges and we set 
\begin{equation}
u_{\varepsilon}(x_i)=\lim_{k\rightarrow \infty}u_h^k(x_i)\quad \forall x_i \in N_h^0
\end{equation}
and extend $u_{\varepsilon}$  without relabeling to $u_{\varepsilon}\in V_h$. Then we have
$u_{\varepsilon}=I_hg$ on $\partial \Omega_h$ and
\begin{equation} \label{91}
T_{\varepsilon}[u_{\varepsilon}](x_i)\ge f(x_i)\quad \forall x_i \in N_h^0.
\end{equation}
We show that even equality holds in (\ref{91}) and assume for it that the inequality in (\ref{91}) is strict 
for a certain $x_i \in N_h^0$. Then we find arbitrary large $k$ such that
 \begin{equation}
T_{\varepsilon}[u_h^k](x_i)>f(x_i).
\end{equation}
But then in the construction of $u_h^{k+1}$ in step (ii) there was a certain 'space' for increasement which contradicts that $(u_h^k(x_i))_k$ is especially
pointwisely a Cauchy sequence.
Hence we have shown existence of $u_{\varepsilon}$ as desired and we also have obtained an a priori $L^{\infty}$-bound
which is independent from the discretization parameters.
\end{proof}

Hereby, our analysis of the discrete model is finished. The remaining sections are concerned with the convergence analysis of the discrete solutions $u_{\varepsilon}$, i.e. they show 
that the discrete solutions $u_{\varepsilon}$ of (\ref{discreteMongeAmpere})
converge to the solution $u$ of (\ref{1}) when $\varepsilon \rightarrow 0$ and the relative size of $h$ and $\delta$ satisfies 
a certain relation provided the original equation satisfies appropriate properties. Recall that $\varepsilon = (h, \delta, \theta)$.

\section{A special auxiliary function}\label{section6}
We 
construct in the following lemma a special auxiliary function.

\begin{lemma} \label{lemma_5_1}
Let $\Omega$ be uniformly convex, $h_0>0$ and $1<E\le c(\Omega, h_0)$ be sufficiently large within this range, $c(\Omega, h_0)$
a suitable constant which depends only on $\Omega$ and $h_0$ with
\begin{equation}
c(\Omega, h_0) \rightarrow \infty 
\end{equation} 
as $h_0 \rightarrow 0$ (this relation becomes more explicit in the proof).
There exists a $h_0=h_0(\Omega)$ such that for all $0<h<h_0$
the following holds. For each node $z\in N_h^0$ and $\delta>0$
with $\dist(z, \partial \Omega_h)\le \delta$ there exists a function $p_h \in V_h$ and $E'>E$
such that $T_{\varepsilon}[p_h](x_i)\ge E'$ for all $x_i\in N_h^0$, $p_h\le 0$ on $\partial \Omega_h$
and 
\begin{equation}
|p_h(z)|\le CE'\delta
\end{equation}
with $C$ depending on $\Omega$.
\end{lemma}
\begin{proof}
Let $z \in N_h^0$ and $\delta>0$ be arbitrary. 
Let $\tilde z \in \partial \Omega$ be a nearest boundary point, i.e.
\begin{equation}
|z-\tilde z| = \dist(z, \partial \Omega).
\end{equation}
Let
\begin{equation}
0< \kappa_1(x) \le ... \le \kappa_n(x)
\end{equation}
be the ordered-by-size $n$ principal curvatures of $\partial \Omega$ in $x\in \partial \Omega$ with 
respect to the outer unit normal of $\partial \Omega$ in $x$ (the convention is here as usual so that e.g. a unit sphere has principal curvatures equal to 1). In view of the uniform convexity of $\partial \Omega$
we have
\begin{equation}
\kappa := \min_{\partial \Omega} \kappa_1 >0.
\end{equation}
For the moment we fix a point $\tilde x \in \mathbb{R}^n$ and  a large $\lambda>0$ and we will adjust them later appropriately.
We define
\begin{equation}
f(x) = e^{\lambda |x-\tilde x|^2}-e^{\lambda |\tilde z-\tilde x|^2}, \quad x \in \mathbb{R}^n.
\end{equation}
The function $f$ looks roughly spoken like a bowl, attains a global minimum in $\tilde x$, is rotationally symmetric around $\tilde x$.
Furthermore, it is strictly monotone increasing and strictly convex along rays starting from $\tilde x$ where in addition this convexity in radial 
direction at a point $x\in \mathbb{R}^n$ can be quantified as being of size $O(\lambda^2e^{\lambda |x-\tilde x|^2})$. Here, the constant hidden in the $O$-notation 
depends on $\Omega$ and $\tilde x$.

Let us now adjust $\tilde x \in \mathbb{R}^n$ where we assume w.l.o.g. that $\tilde x \notin \bar \Omega$ and choose $R>\frac{1}{\kappa}$ such that
\begin{equation}
\partial \Omega \cap \partial B_R(\tilde x) = \{\tilde z\} \quad \text{and} \quad
\Omega \subset B_R(\tilde x).
\end{equation} 
Let 
\begin{equation}
c_1 = \max_{\bar \Omega}|x-\tilde x|, \quad c_2 = \min_{\bar \Omega}|x-\tilde x|�
\end{equation}
then the second derivatives of $f$ in $\bar \Omega$ are of size at least $O(\lambda^2c_2^2e^{\lambda c_2^2})$
and the first derivatives are of size at most $O(c_1\lambda e^{\lambda c_1})$. We increase $\lambda$
until $O(\lambda^2c_2^2e^{\lambda c_2^2})$ is large  compared to $E$ and $O(c_1\lambda e^{\lambda c_1})$. Then we set 
\begin{equation}
E' = \max\{O(\lambda^2c_2^2e^{\lambda c_2^2}), O(c_1\lambda e^{\lambda c_1})\}
\end{equation}
as well as
\begin{equation}
p_h=I_h(f-f(z)).
\end{equation}
Now we may assume that $E$ and $E'$ are bounded by a constant which may become arbitrary large provided  $h_0(\Omega)$
is correspondingly small
so that for $0<h<h_0$ the previous interpolation does not produce relevant errors. This finishes the proof of the lemma.
Note that we tacitly introduced concrete but generic constants in order to replace the $O$-notation
in the context of inequalities.
\end{proof}

\section{An approximating problem}\label{section7}
In the following remark we formulate some rather weak assumptions for equation \eqref{1} and its solution which allow us to 
construct an approximating smooth problem with a smooth solution (which is also the main purpose of this section). 

\begin{remark} \label{regularity_assumption}
\begin{enumerate}
\item (Regularity) $u\in C^{1}(\bar \Omega)$ is a viscosity solution of \eqref{1}.
\item (Comparison principle one sided around the solution) There is $\varepsilon_1>0$ such that the following holds. Given continuous $\tilde f_1, \tilde f_2, \tilde g_1, \tilde g_2$ with $f\le \tilde f_2\le \tilde f_1 \le f+\varepsilon_1$ in $\Omega$ and
$g-\varepsilon_1\le \tilde g_1\le \tilde g_2 \le g$ on $\partial \Omega$ and continuous viscosity solutions $\tilde u_1$
and $\tilde u_2$ of \eqref{1} with respect to the data $\tilde f_1, \tilde g_1$ and $\tilde f_2, \tilde g_2$, respectively, then there holds a comparison principle in the usual sense, i.e. $\tilde u_1 \le \tilde u_2$ in $\Omega$.
\end{enumerate}
\end{remark}
Let $\Omega_n \supset \Omega$, $n\in \mathbb{N}$, be an approximation of $\Omega$ by smooth convex sets with respect to the Hausdorff distance $d_H$, i.e. 
\begin{equation}
0< \dist_H(\Omega, \Omega_n) \le \delta_n \rightarrow 0.
\end{equation}
Let $p\in \Omega$ be arbitrary and fixed. The family of rays
\begin{equation}
\{R_p=\{p+te: t\ge 0\}:e\in \mathbb{R}^n, \|e\|=1\}
\end{equation}
clearly defines a bijection 
\begin{equation}
b_n: \partial \Omega \rightarrow \partial \Omega_n
\end{equation}
by mapping $R_p \cap \partial \Omega$ to $R_p \cap \partial \Omega_n$.
Let $f_n$ and $g_n$ be smooth functions in $\mathbb{R}^n$ approximating  $f$ and $g$, respectively,
such that
\begin{equation}\label{inequality}
f<f_n, \quad g_n<g,
\end{equation}
\begin{equation}
|g(x)-g_n(b_n(x))|\le \delta_n, \quad |f(x)-f_n(y)|\le \delta_n
\end{equation}
for all 
\begin{equation}
x \in \partial \Omega, y \in [x, b_n(x)]=\{tx+(1-t)b_n(x): 0\le t \le 1\}
\end{equation}
and
\begin{equation}
|f(x)-f_n(x)|\le \delta_n, \quad x \in \bar \Omega.
\end{equation}
We note that the above approximations can be obtained in a standard fashion and indicate that
especially inequalites (\ref{inequality}) can be achieved by first replacing $f$ by $f+\frac{\delta_n}{2}$
and $g$ by $g-\frac{\delta_n}{2}$ and then extending and mollifying these modified functions.

Let $u_n\in C^{\infty}(\bar \Omega)$ be classical solutions of
\begin{equation} \label{Approx_1}
\det \left(D^2u_n-A(x, Du_n)\right) = f_n \quad \text{in }\Omega_n
\end{equation}
and
\begin{equation} \label{Approx_2}
u_n=g_n \quad \text{on }\partial \Omega_n,
\end{equation}
cf. the useful exposition in the introduction of \cite{FigalliKimYoung-HeonMcCann2011} for an overview of different assumptions and corresponding references leading to different regularities.
Here, we mention especially the reference \cite{LiuTrudingerWang2010} mentioned on page 2 
of \cite{FigalliKimYoung-HeonMcCann2011} for the smooth case: smooth data imply smoothness of the solution.
 In the next lemma we estimate $u_n$ on the boundary $\partial \Omega$ by constructing suitable barriers.
We have the following plausible lemma which we will also prove rigorously in the following
without using any a priori estimates for the solution. Our proof without using the last named type of  estimates has the advantage that the lemma also holds when only the sufficient regularity without a priori estimates is available.
\begin{lemma}
There holds 
\begin{equation}
|u_n-g|\rightarrow 0
\end{equation}
uniformly on $\partial \Omega$ as $n \rightarrow \infty$.
\end{lemma}
\begin{proof}
Let us fix $z \in \partial \Omega$ and evaluate 
$g$ and $u_n$ at $z$ and compare them. 
Let $y$
be the closest point to $z$ in  $\partial \Omega_n$, then $|z-y|\le \delta_n$ and given $\delta >0$
we have
\begin{equation}
|g(z)-g_m(y)|\le |g(z)-g_m(z)|+|g_m(z)-g_m(y)|\le \delta 
\end{equation}
provided $m$ is sufficiently large and also $n=n(m)$ is sufficiently large.
Let $p$ be the (not discrete) barrier function from the proof of Lemma \ref{lemma_5_1}
associated with $\Omega_n$, $z\in \Omega_n$, i.e.
\begin{equation}
p(x) = e^{\lambda |x-\tilde x|^2}-e^{\lambda |\tilde z-\tilde x|^2}
\end{equation}
where $\tilde x, \tilde z$ are chosen accordingly to the proof of Lemma \ref{lemma_5_1} and we may arrange it so that $\tilde z$
equals the above specified $y$, i.e. $y=\tilde z$. 
We define the function
\begin{equation}
b_m^{-}:= p(x)+g_m(y)-C_0|x-y|
\end{equation}
where $C_0\ge \|g_m\|_{C^1(\bar \Omega)}$.
Clearly, $b_m^{-}\le g_m$ in $\bar \Omega$ in view of $p\le 0$ in $\bar \Omega$ and we also have that
\begin{equation}
T_{\varepsilon}[b_m^{-}]\ge f_m\quad \text{ in } \Omega_n
\end{equation}
provided $\lambda$ is sufficiently large. Hence by the comparison principle we conclude that
\begin{equation}
b_m^{-}\le u_m \quad \text{ in } \Omega_n.
\end{equation}
Evaluating this inequality in $z$ and retranslation by using the definition of $b_m^{-}$ leads to
\begin{equation}
g_m(y)-C_m\delta_n \le u_m(z)
\end{equation}
where $C_m>0$ is a constant which may depend on $m$ (but not on $n$).
Similarly, using 
\begin{equation}
b_m^{+}(x):=-p(x)+g_m(y)+C_0|x-y|
\end{equation}
as upper barrier function for $u_m$ which is $Q$-convex on the one-dimensional line 
\begin{equation}
\bar \Omega \cap \{x_1=0\}
\end{equation}
we conclude from the maximum principle in one variable that $b_m^{+}\ge u_m$ in $\bar \Omega$.
Similarly as before we then get
\begin{equation}
u_m(z) \le g_m(y)+C_m\delta_z.
\end{equation}
Putting this by using the triangle inequality together we conclude that
\begin{equation}
|g(z)-u_m(z)| \le |g(z)-g_m(y)|+|g_m(y)-u_m(z)| \le C_m\delta_n+\delta.
\end{equation}
\end{proof}

The following lemma gives the desired arbitrary good approximation of \eqref{1} by smooth problems (i.e. with smooth data) with smooth solutions. 
\begin{lemma} \label{lemma7}
Let $f_n$, $g_n$, $\Omega_n$ and $u_n$ as before. Let $\varepsilon>0$ then there is $n \in \mathbb{N}$ such that
\begin{equation}
|u_n-u|\le \varepsilon
\end{equation}
in $\Omega$.
\end{lemma}

\begin{proof}
Let $q\le 0$ be the function from (\ref{ansatz_function1}) and $\alpha, \beta>0$ 
suitable constants which will be specified later.
We consider the auxiliary function
\begin{equation}
u^{-}:=u+\alpha q-\beta.
\end{equation}
We observe that
\begin{equation}
u^{-}\le u-\beta = g -\|g-g_n\|_{L^{\infty}(\partial \Omega)} \le g_n
\end{equation}
on $\partial \Omega$ for $\beta= \|g-g_n\|_{L^{\infty}(\partial \Omega)}$.
Let $\phi \in C^2(\Omega)$ and $x_0 \in \Omega$
be a point where 
\begin{equation}
u^{-}-\phi=u-(\phi-\alpha q+\beta)
\end{equation}
attains a maximum. Abbreviating 
\begin{equation}
w=\phi - \alpha q + \beta \in C^2(\Omega)
\end{equation}
and using that $u$ is a viscosity subsolution of \eqref{1} we conclude that
\begin{equation}
T[w]\ge f.
\end{equation}
We would like to show that 
\begin{equation} \label{goal_to_show}
T[\phi]\ge f_n
\end{equation}
from which we deduce that $u^{-}$ is a viscosity subsolution of the problem \eqref{Approx_1}, \eqref{Approx_2}.
For it we evaluate $T[\phi]$ more explicitly. As a tool we use the following straightforward and general relation. For positive numbers $a_1, ..., a_n, \varepsilon$ holds when setting
\begin{equation}
\prod_{i=1}^na_i = z>0
\end{equation} 
that
\begin{equation} \label{deliberation}
\prod_{i=1}^n(a_i+\varepsilon)\ge \prod_{i=1}^na_i + \varepsilon^{n-1} \sum_{i=1}^na_i\ge z+\varepsilon^{n-1} z^{\frac{1}{n}}.
\end{equation}
For fixed $x\in \bar \Omega$ we let $a_1, ..., a_n$ be the eigenvalues of 
\begin{equation}
D^2w(x)-A(x, Dw(x)).
\end{equation}
From the min-max characterization of eigenvalues (given by the Courant-Fisher-Weyl maximum
principle)  we conclude that the ordered by size eigenvalues $\lambda_1 \le ... \le \lambda_n$ of 
\begin{equation}
D^2\phi(x)-A(x, D\phi(x))
\end{equation}
satisfy 
\begin{equation}
\lambda_i \ge a_i+\varepsilon
\end{equation}
for some $\varepsilon>0$ provided $\lambda$ in the definition of $p$ is sufficiently large. Hence we have
\begin{equation}
T[\phi] \ge f+\varepsilon^{n-1} \min f^{\frac{1}{n}}
\end{equation}
in view of our previous deliberation (\ref{deliberation}).
Clearly, we can achieve that \eqref{goal_to_show} holds. Since $u\in C^1(\bar \Omega)$ we may assume w.l.o.g. in the previous argumentation that $\|\phi\|_{C^1(\bar \Omega)}\le c(\|u\|_{C^1(\bar \Omega)})$. Hence the previous mechanism works
for $\lambda$ sufficiently large depending only on $f$, $g$ and $\|u\|_{C^1(\bar \Omega)}$ and independently from the choice of $\alpha$. Hence we see that for $n$ sufficiently large we may choose $\alpha$ sufficiently small and the claim follows since
\begin{equation}
u^{-} \le u_n \le u
\end{equation}
and 
\begin{equation}
u-u^{-} = -\alpha q+\beta
\end{equation}
can be made small for large $n$.
\end{proof}

\section{Convergence properties of the discrete solutions when the scales go to zero}
\label{section8}

Since $u_{\varepsilon}$ is defined in the computational domain $\Omega_h$
 and $\Omega_h\subset \Omega$, we extend $u_{\varepsilon}$ to $\Omega$
as follows. 
Given $x\in \Omega\setminus \Omega_h$ we choose $z \in \partial \Omega_h$
as the nearest point in $\Omega_h$ to $x$ which is unique because $\Omega_h$ is convex and let
\begin{equation}\label{boundary_definition}
u_{\varepsilon}(x):= u_{\varepsilon}(z)=I_hg(z) \quad \forall x \in \Omega\setminus \Omega_h.
\end{equation}

In the following theorem we prove convergence of the discrete solutions to the solution of the original problem.

 \begin{theorem}
 Let $\Omega$ be uniformly convex, $f, g \in C(\bar \Omega)$
 and $f> 0$ in $\bar \Omega$. Let $u$ be a solution of \eqref{1} satisfying the assumptions in Remark \ref{regularity_assumption}. 
 The discrete solutions $u_{\varepsilon}$
 of \eqref{equ_1} and \eqref{boundary_definition} converge uniformly to $u$ as $\varepsilon=(h, \delta, \theta)\rightarrow 0$ and $\frac{h}{\delta}\rightarrow 0$.
 Here, the constant $\Lambda$ in the definition of the finite element space, cf. (\ref{def_V_h}),  depends on $\varepsilon$ in the general case. 
 If in addition the sequence of solutions $u_n$ of the approximating problems as constructed in the previous section
is uniformly bounded in $C^3$ then $\Lambda$ can be chosen uniformly in $\varepsilon$.
\end{theorem}

\begin{proof}
We first split the domain 
\begin{equation}
\begin{aligned}
\|u-u_{\varepsilon}\|_{L^{\infty}(\Omega)}
\le& \| u-u_{\varepsilon}\|_{L^{\infty}(\Omega_h)}+\|u-u_{\varepsilon}\|_{L^{\infty}(\Omega\setminus \Omega_h)} \\
=& I_1 + I_2.
\end{aligned}
\end{equation}
Estimating the first term with the triangle inequality gives
\begin{equation}
I_1 \le \|u-u_{n}\|_{L^{\infty}(\Omega_h)}+\|u_n-I_hu_n\|_{L^{\infty}(\Omega_h)}
+\|I_h u_n-u_{\varepsilon}\|_{L^{\infty}(\Omega_h)}
\end{equation}
where $u_n$ is the solution of the approximating problem from the previous section and $n$ is assumed to be sufficiently large,
and hence 
\begin{equation}
\|u-u_n\|_{L^{\infty}(\Omega_h)}
\end{equation}
can be assumed to be arbitrarily small.
In view of the standard interpolation estimate 
\begin{equation}
 \|u_n-I_hu_n\|_{L^{\infty}(\Omega_h)}\le c \|u_n\|_{W^{2, \infty}(\Omega)}h^2
\end{equation}
we may assume that $h=h(n)$ is so small that the norm on the left-hand side is as small as desired as well as that $\Lambda=\Lambda(\varepsilon)$ is sufficiently large.
From (\ref{boundary_definition}) we conclude that for all $x\in \Omega \setminus \Omega_h$
and corresponding $z=z(x)\in  \partial \Omega_h$ we have
\begin{equation}
\begin{aligned}
|u(x)-u_{\varepsilon}(x)| =& |u(x)-u_{\varepsilon}(z)|\\
\le& |u(x)-u(z)|+|u(z)-u_{\varepsilon}(z)|.
\end{aligned}
\end{equation}
Denoting the modulus of continuity of $u \in C(\bar \Omega)$
by $\tau$ we have
\begin{equation}
I_2= \|u-u_{\varepsilon}\|_{L^{\infty}(\Omega\setminus \Omega_h)}
\le \tau(\dist_H(\Omega, \Omega_h))+\|u-u_{\varepsilon}\|_{L^{\infty}(\Omega_h)}.
\end{equation}
Since $\dist_H(\Omega, \Omega_h)\rightarrow 0$ as $h\rightarrow 0$
the proof reduces to showing that  
\begin{equation}
\|I_h u_n-u_{\varepsilon}\|_{L^{\infty}(\Omega_h)}
\end{equation}
can be made arbitrarily small which will be shown in the remaining part of the proof. Note that instead of arguing with the modulus of continuity of $u$
we could have also used the $C^{0}$-estimates which we derived in the proof of Lemma \ref{lemma7} and the modulus of continuity of the corresponding approximating $u_n$.

Recall that we have chosen and will choose for the following $n$ sufficiently large. Furthermore, we will assume that $h=h(n)$
is chosen sufficiently small and $\Lambda=\Lambda(\varepsilon)$ sufficiently large.

We use the function $q_h=I_hq$ where
\begin{equation}
q(x) = e^{\lambda |x-\tilde x|^2}-R
\end{equation}
with $\tilde x$ outside $\bar \Omega$ and $R>0$ so that $q < 0$ in $\bar \Omega$.
We define the discrete lower barrier as
\begin{equation}
b_{\varepsilon}^{-}=u_{\varepsilon}+\rho q_h
\end{equation}
where $\rho>0$ so that 
\begin{equation}
b_{\varepsilon}^{-} \le g_n
\end{equation} 
on $\partial \Omega_h$. W.l.o.g. let us assume that
\begin{equation}
T_{\varepsilon}[I_hu_n] \le f_n+\|f-f_n\|_{L^{\infty}(\bar \Omega)}+\frac{1}{n}.
\end{equation}
Choosing $\lambda>0$ sufficiently large we achieve that
\begin{equation}
T_{\varepsilon}[b_{\varepsilon}^{-}] \ge T_{\varepsilon}[I_hu_n]
\end{equation}
and hence
\begin{equation}
b_{\varepsilon}^{-} \le I_hu_n + O(h).
\end{equation}
A similar argument with $b_{\varepsilon}^{+}:= u_{\varepsilon}-\rho q_h$, $\rho>0$ suitable,
results in $b_{\varepsilon}^{+}\ge I_hu_n-O(h)$. 

Clearly, this leads summarized to 
\begin{equation}
|I_hu_n-u_{\varepsilon}| \le -2\rho q_h+O(h).
\end{equation}
Now, choosing $\varepsilon$ (resp. $h$) small, $\Lambda$ sufficiently large, and $\rho$, $\lambda$ suitable (not depending on $h$ or $\Lambda$) we get the desired convergence.
This completes the proof.
\end{proof}




\end{document}